\documentclass[12pt]{article}
\pdfoutput=1
\usepackage{amssymb}
\usepackage{amsthm}
\usepackage{amsmath}
\usepackage{latexsym}
\usepackage[english]{babel}
\usepackage{graphicx}
\usepackage[justification=centering, labelfont=bf]{caption}
\usepackage{mathtools}
\usepackage[hidelinks]{hyperref}
\usepackage{enumitem}
\usepackage{mathrsfs}
\usepackage{tikz}
\usepackage[shortcuts]{extdash}

\usetikzlibrary{shapes,snakes}
\usetikzlibrary{arrows.meta}

\usepackage[left=1.9cm,right=1.9cm,top=1.9cm,bottom=1.9cm,bindingoffset=0cm]{geometry}

\title{On differences between DP-coloring and list coloring}
\date{}

\author{Anton~Bernshteyn\thanks{Department of Mathematics, University of Illinois at Urbana--Champaign, IL, USA, \href{bernsht2@illinois.edu}{\texttt{bernsht2@illinois.edu}}. Research of this author is partially supported by the Illinois Distinguished Fellowship.} \and Alexandr~Kostochka\thanks{Department of Mathematics, University of Illinois at Urbana--Champaign, IL, USA and Sobolev Institute of Mathematics, Novosibirsk, Russia, \href{mailto:kostochk@math.uiuc.edu}{\texttt{kostochk@math.uiuc.edu}}. Research of this author is supported in part by NSF grant DMS-1600592 and grants 15-01-05867 and 16-01-00499 of the Russian Foundation for Basic Research.}}

\newtheorem{theo}{Theorem}[section]

\newtheorem{lemma}[theo]{Lemma}

\newtheorem{conj}[theo]{Conjecture}

\newtheorem{problem}[theo]{Problem}

\newcommand*{\myproofname}{Proof}

\theoremstyle{definition}
\newtheorem{defn}[theo]{Definition}
\newtheorem{exmp}[theo]{Example}

\newtheorem{remk}[theo]{Remark}

\newcommand{\0}{\varnothing}
\newcommand{\set}[1]{\{#1\}}

\newcommand{\N}{\mathbb{N}}
\newcommand{\Z}{\mathbb{Z}}

\newcommand{\Cov}[1]{\mathscr{#1}}

\renewcommand{\epsilon}{\varepsilon}
\renewcommand{\phi}{\varphi}
\renewcommand{\leq}{\leqslant}
\renewcommand{\geq}{\geqslant}
\newcommand{\powerset}[1]{\operatorname{Pow}(#1)}

\newcommand{\defeq}{\coloneqq}

\numberwithin{equation}{section}

\makeatletter
\newcommand{\neutralize}[1]{\expandafter\let\csname c@#1\endcsname\count@}
\makeatother

\begin{document}
	
	\maketitle
	
	\begin{abstract}
		\noindent DP-coloring (also known as correspondence coloring) is a generalization of list coloring introduced recently by Dvo\v r\' ak and Postle~\cite{DP15}. Many known upper bounds for the list-chromatic number extend to the DP-chromatic number, but not all of them do. In this  note we describe some  properties of DP-coloring that set it aside from list coloring. In particular, we give an example of a planar bipartite graph with DP-chromatic number $4$ and prove that the edge-DP-chromatic number of a $d$-regular graph with $d\geq 2$ is always at least $d+1$.
	\end{abstract}
	
	\section{Introduction}
	
	\subsection{Basic notation and conventions}
	
	We use $\N$ to denote the set of all nonnegative integers. 
	For a set $S$, $\powerset{S}$ denotes the power set of $S$, i.e., the set of all subsets of $S$. All graphs considered here are finite, undirected, and simple, except in Section~\ref{sec:multi}, which mentions (loopless) multigraphs. For a graph~$G$, $V(G)$ and $E(G)$ denote the vertex and the edge sets of $G$ respectively. For a subset $U \subseteq V(G)$, $G[U]$ is the subgraph of $G$ induced by $U$. For two subsets $U_1$, $U_2 \subseteq V(G)$, $E_G(U_1, U_2) \subseteq E(G)$ is the set of all edges of $G$ with one endpoint in $U_1$ and the other one in $U_2$. The maximum degree of $G$ is denoted by $\Delta(G)$. 
	
	\subsection{Graph coloring, list coloring, and DP-coloring}
	
	Recall that a \emph{proper coloring} of a graph $G$ is a function $f \colon V(G) \to C$, where $C$ is a set of \emph{colors}, such that  $f(u) \neq f(v)$ for each edge $uv \in E(G)$. The \emph{chromatic number} $\chi(G)$ of $G$ is the smallest $k \in \N$ such that there exists a proper coloring $f \colon V(G) \to C$ with $|C| = k$.
	
	\emph{List coloring} is a generalization of ordinary graph coloring that was introduced independently by Vizing~\cite{Viz76} and Erd\H{o}s, Rubin, and Taylor~\cite{ERT79}. As in the case of ordinary graph coloring, let $C$ be a set of \emph{colors}. A \emph{list assignment} for a graph $G$ is a function $L \colon V(G) \to \powerset{C}$; if $|L(u)| = k$ for all $u \in V(G)$, then~$L$ is called a \emph{$k$-list assignment}. A proper coloring $f \colon V(G) \to C$ is called an \emph{$L$-coloring} if  $f(u) \in L(u)$ for each $u \in V(G)$. The \emph{list-chromatic number} $\chi_\ell(G)$ of $G$  is the smallest $k \in \N$ such that $G$ admits an $L$-coloring for every $k$-list assignment $L$ for $G$. An immediate consequence of this definition is that $\chi_\ell(G) \geq \chi(G)$ for all graphs~$G$, since ordinary coloring is the same as $L$-coloring with $L(u) = C$ for all $u \in V(G)$. On the other hand, it is well-known that the gap between $\chi(G)$ and $\chi_\ell(G)$ can be arbitrarily large; for instance, $\chi(K_{n, n}) = 2$, while $\chi_\ell(K_{n,n}) = (1 + o(1))\log_2(n) \to \infty$ as $n \to \infty$, where $K_{n,n}$ denotes the complete bipartite graph with both parts having size $n$.
	
	In this paper we study a further generalization of list coloring that was recently introduced by Dvo\v r\' ak and Postle~\cite{DP15}; they called it \emph{correspondence coloring}, and we call it \emph{DP-coloring} for short. In the setting of DP-coloring, not only does each vertex get its own list of available colors, but also the identifications between the colors in the lists can vary from edge to edge.
	
	\begin{defn}\label{defn:cover}
		Let $G$ be a graph. A \emph{cover} of $G$ is a pair $\Cov{H} = (L, H)$, consisting of a graph $H$ and a function $L \colon V(G) \to \powerset{V(H)}$, satisfying the following requirements:
		\begin{enumerate}[labelindent=\parindent,leftmargin=*,label=(C\arabic*)]
			\item the sets $\set{L(u) \,:\,u \in V(G)}$ form a partition of $V(H)$;
			\item for every $u \in V(G)$, the graph $H[L(u)]$ is complete;
			\item if $E_H(L(u), L(v)) \neq \0$, then either $u = v$ or $uv \in E(G)$;
			\item \label{item:matching} if $uv \in E(G)$, then $E_H(L(u), L(v))$ is a matching.
		\end{enumerate}
		A cover $\Cov{H} = (L, H)$ of $G$ is \emph{$k$-fold} if $|L(u)| = k$ for all $u \in V(G)$.
	\end{defn}
	
	\begin{remk}
		The matching $E_H(L(u), L(v))$ in Definition~\ref{defn:cover}\ref{item:matching} does not have to be perfect and, in particular, is allowed to be empty.
	\end{remk}

	
	
	\begin{defn}
		Let $G$ be a graph and let $\Cov{H} = (L, H)$ be a cover of $G$. An \emph{$\Cov{H}$-coloring} of $G$ is an independent set in $H$ of size $|V(G)|$.
	\end{defn}
	
	\begin{remk}\label{remk:single}
		By definition, if $\Cov{H} = (L, H)$ is a cover of $G$, then $\set{L(u)\,:\, u \in V(G)}$ is a partition of~$H$ into $|V(G)|$ cliques. Therefore, an independent set $I \subseteq V(H)$ is an $\Cov{H}$-coloring of $G$ if and only if $|I \cap L(u)| = 1$ for all $u \in V(G)$.
	\end{remk}
	
	\begin{defn}
		Let $G$ be a graph. The \emph{DP-chromatic number} $\chi_{DP}(G)$ of $G$ is the smallest $k \in \N$ such that $G$ admits an $\Cov{H}$-coloring for every $k$-fold cover $\Cov{H}$ of $G$.
	\end{defn}
	
	\begin{exmp}\label{exmp:cycles}
	Figure~\ref{fig:cycle} shows two distinct $2$-fold covers of the $4$-cycle $C_4$. Note that $C_4$ admits an $\Cov{H}_1$-coloring but not an $\Cov{H}_2$-coloring. In particular, $\chi_{DP}(C_4) \geq 3$; on the other hand, it can be easily seen that $\chi_{DP}(G) \leq \Delta(G) + 1$ for any graph $G$, and so we have $\chi_{DP}(C_4) = 3$. A similar argument demonstrates that $\chi_{DP}(C_n) = 3$ for any cycle $C_n$ of length $n \geq 3$.
	\end{exmp}
	
	\begin{figure}[h]
		\centering	
		\begin{tikzpicture}[scale=0.7]
		\definecolor{light-gray}{gray}{0.95}

		\filldraw[fill=light-gray]
		(6.5,0) circle [x radius=1cm, y radius=5mm, rotate=45]
		(6.5,3) circle [x radius=1cm, y radius=5mm, rotate=-45]
		(9.5,0) circle [x radius=1cm, y radius=5mm, rotate=-45]
		(9.5,3) circle [x radius=1cm, y radius=5mm, rotate=45];
		
		
		\foreach \x in {(6, -0.5), (6, 3.5), (7, 0.5), (7, 2.5), (9, 0.5), (9, 2.5), (10, -0.5), (10, 3.5)}
		\filldraw \x circle (4pt);
		
		\draw[thick] (6, -0.5) -- (6, 3.5) -- (10, 3.5) -- (10, -0.5) -- cycle;
		
		\draw[thick] (7, 0.5) -- (7, 2.5) -- (9, 2.5) -- (9, 0.5) -- cycle;
		
		\draw[thick] (6, -0.5) -- (7, 0.5) (6, 3.5) -- (7, 2.5) (9, 2.5) -- (10, 3.5) (10, -0.5) -- (9, 0.5);
		
		\node at (8, -1.5) {$\Cov{H}_1$};
		
		\filldraw[fill=light-gray]
		(13+3,0) circle [x radius=1cm, y radius=5mm, rotate=45]
		(13+3,3) circle [x radius=1cm, y radius=5mm, rotate=-45]
		(16+3,0) circle [x radius=1cm, y radius=5mm, rotate=-45]
		(16+3,3) circle [x radius=1cm, y radius=5mm, rotate=45];
		
		
		\foreach \x in {(12.5+3, -0.5), (12.5+3, 3.5), (13.5+3, 0.5), (13.5+3, 2.5), (15.5+3, 0.5), (15.5+3, 2.5), (16.5+3, -0.5), (16.5+3, 3.5)}
		\filldraw \x circle (4pt);
		
		\draw[thick] (12.5+3, -0.5) -- (12.5+3, 3.5) -- (16.5+3, 3.5) -- (16.5+3, -0.5) -- (13.5+3, 0.5) -- (13.5+3, 2.5) -- (15.5+3, 2.5) -- (15.5+3, 0.5) -- cycle;
		
		\draw[thick] (12.5+3, -0.5) -- (13.5+3, 0.5) (12.5+3, 3.5) -- (13.5+3, 2.5) (15.5+3, 0.5) -- (16.5+3, -0.5) (15.5+3, 2.5) -- (16.5+3, 3.5);
		
		\node at (14.5+3, -1.5) {$\Cov{H}_2$};
		\end{tikzpicture}
		\caption{Two distinct $2$-fold covers of a $4$-cycle.
		}\label{fig:cycle}
	\end{figure}
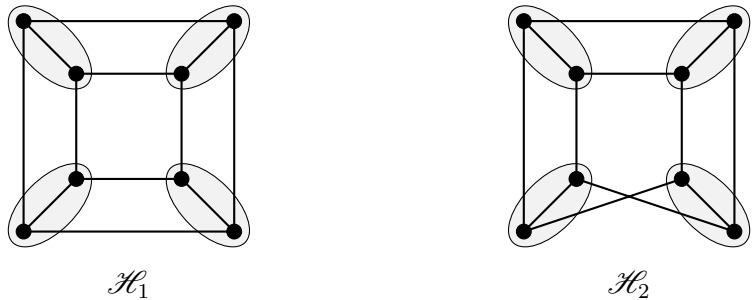
	
	One can construct a cover of a graph $G$ based on a list assignment for~$G$, thus showing that list coloring is a special case of DP-coloring and, in particular, $\chi_{DP}(G) \geq \chi_\ell(G)$ for all graphs $G$.
	
	\begin{figure}[h]
		\centering
		\begin{tikzpicture}[scale=0.7,every text node part/.style={align=left}]
		\draw[thick] (-7,0) node[anchor=east] (u1) {$u_1$} -- (-5, 2) node[anchor=south] {$u_2$} -- (-3, 0) node[anchor=west] (u3) {$u_3$} -- (-5, -2) node[anchor=north] {$u_4$} -- cycle (u1) -- (u3);
		
		\foreach \x in {(-7,0), (-5,2), (-3, 0), (-5, -2)}
		\filldraw \x circle (4pt);
		
		\node at (0,1) {$L(u_1) = \set{1,2}$\\$L(u_2) = \set{1,3}$\\$L(u_3) = \set{1,2}$\\$L(u_4) = \set{2,3}$};
		
		
		\definecolor{light-gray}{gray}{0.95}
		
		\begin{scope}[xshift=2cm]
		
		\filldraw[fill=light-gray] (3+3,0) circle [x radius=0.3cm, y radius=0.8cm] (5+3,2) circle [x radius=0.3cm, y radius=0.8cm] (7+3,0) circle [x radius=0.3cm, y radius=0.8cm] (5+3,-2) circle [x radius=0.3cm, y radius=0.8cm];
		
		
		
		
		
		\foreach \x in {(3+3,0.5), (3+3, -0.5), (5+3,1.5), (5+3, 2.5), (5+3,-1.5), (5+3, -2.5), (7+3, 0.5), (7+3, -0.5)}
		\filldraw \x circle (4pt);
		
		\draw[thick] (3+3, 0.5)  -- (3+3, -0.5)  -- (5+3, -1.5) -- (7+3, -0.5) -- (7+3, 0.5) -- (5+3, 1.5)  -- cycle -- (7+3, 0.5) (5+3, 2.5) -- (5+3, 1.5) (5+3, -1.5) -- (5+3, -2.5) (3+3, -0.5) -- (7+3, -0.5);
		
		\node at (3.5,0) {$L'(u_1) = \left\{\setlength\arraycolsep{0pt}\begin{array}{l} (u_1, 1) \\ (u_1, 2)\end{array}\right.$};
		\node at (5.5,2) {$L'(u_2) = \left\{\setlength\arraycolsep{0pt}\begin{array}{l} (u_2, 3) \\ (u_2, 1)\end{array}\right.$};
		\node at (12.5,0) {$\left.\setlength\arraycolsep{0pt}\begin{array}{l} (u_3, 1) \\ (u_3, 2)\end{array}\right\} = L'(u_3)$};
		\node at (10.5,-2) {$\left.\setlength\arraycolsep{0pt}\begin{array}{l} (u_4, 2) \\ (u_4, 3)\end{array}\right\} = L'(u_4)$};
		\end{scope}
		
		\end{tikzpicture}
		\caption{A graph with a $2$-list assignment and the corresponding $2$-fold cover.}\label{fig:list}
	\end{figure}
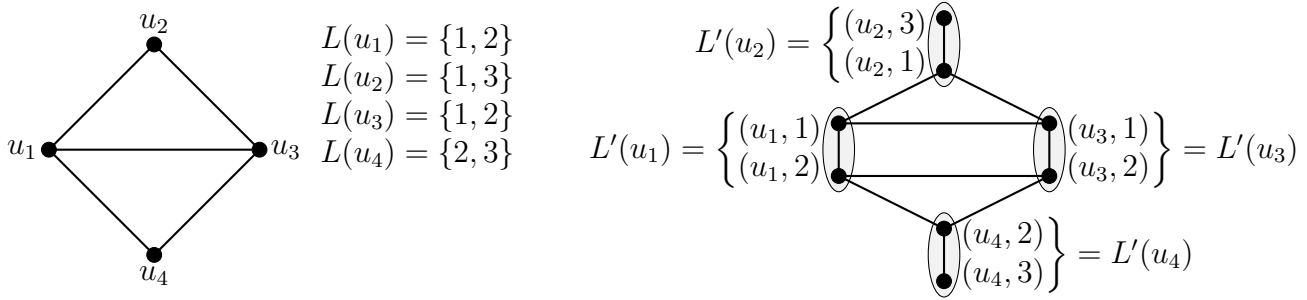
	
	More precisely, let $G$ be a graph and suppose that $L \colon V(G) \to \powerset{C}$ is a list assignment for~$G$, where $C$ is a set of colors. Let $H$ be the graph with vertex set
	$$
		V(H) \defeq \set{(u, c)\,:\, u \in V(G) \text{ and } c \in L(u)},
	$$
	in which two distinct vertices $(u, c)$ and $(v, d)$ are adjacent if and only if
	\begin{itemize}
		\item[--] either $u = v$,
		\item[--] or else, $uv \in E(G)$ and $c = d$.
	\end{itemize}
	For each $u \in V(G)$, set
	$$
		L'(u) \defeq \set{(u, c) \,:\, c \in L(u)}.
	$$
	Then $\Cov{H} \defeq (L', H)$ is a cover of $G$, and there is a natural bijective correspondence between the $L$-colorings and the $\Cov{H}$-colorings of $G$. Indeed, if $f \colon V(G) \to C$ is an $L$-coloring of $G$, then the set
	$$
		I_f \defeq \set{(u, f(u)) \,:\, u \in V(G)}
	$$
	is an $\Cov{H}$-coloring of $G$. Conversely, given an $\Cov{H}$-coloring $I \subseteq V(H)$ of $G$, $|I \cap L'(u)| = 1$ for all $u \in V(G)$, so one can define an $L$-coloring $f_I \colon V(G) \to C$ by the property $$(u, f_I(u)) \in I \cap L'(u)$$ for all $u \in V(G)$.

	\subsection{DP-coloring vs. list coloring and the results of this note}
	
	Some upper bounds on list-chromatic number hold for DP-chromatic number as well. For instance, it is easy to see that $\chi_{DP}(G) \leq d +1$ for any $d$-degenerate graph $G$. Dvo\v r\' ak and Postle~\cite{DP15} observed that for any planar graph $G$, $\chi_{DP}(G) \leq 5$ and, moreover, $\chi_{DP}(G) \leq 3$ if $G$ is a planar graph of girth at least $5$ (these statements are extensions of classical results of Thomassen~\cite{Tho94, Tho95} on list colorings).
	
	Furthermore, there are statements about list coloring whose only known proofs involve DP\-/coloring in essential ways. For example, the reason why Dvo\v r\' ak and Postle originally introduced DP\-/coloring was to prove that every planar graph without cycles of lengths $4$ to $8$ is $3$-list-colorable~\cite[Theorem~1]{DP15}, thus answering a long-standing question of Borodin~\cite[Problem~8.1]{Bor13}. Another example can be found in~\cite{BK16}, where Dirac's theorem on the minimum number of edges in critical graphs~\cite{Dir57, Dir74} is extended to the framework of DP-colorings, yielding a solution to the problem, posed by Kostochka and Stiebitz~\cite{KS02}, of classifying list-critical graphs that satisfy Dirac's bound with equality.
	
	On the other hand, DP-coloring and list coloring are also strikingly different in some respects. For instance, Bernshteyn~\cite[Theorem~1.6]{Ber16} showed that the DP-chromatic number of every graph with average degree $d$ is $\Omega(d/\log d)$, i.e., close to linear in $d$. Recall that due to a celebrated result of Alon~\cite{Alo00}, the list-chromatic number of such graphs is $\Omega(\log d)$, and this bound is sharp for ``small'' bipartite graphs. In spite of this, known upper bounds on list-chromatic numbers often have the same order of magnitude as in the DP\-/coloring setting. For example, by Johansson's theorem~\cite{Joh}, triangle-free graphs $G$ of maximum degree~$\Delta$ satisfy $\chi_\ell(G) = O(\Delta/\log\Delta)$. The same asymptotic upper bound holds for $\chi_{DP}(G)$~\cite[Theorem~1.7]{Ber16}. Recently, Molloy~\cite{Mol} refined Johansson's result to $\chi_\ell(G) \leq (1 + o(1)) \Delta/\ln\Delta$, and this improved bound, including the constant factor, also generalizes to DP-colorings~\cite{BerJM}. 
	
	 Important tools in the study of list coloring that do not generalize to the framework of DP\-/coloring are the orientation theorems of Alon and Tarsi~\cite{AT92} and the closely related Bondy--Boppana--Siegel lemma (see~\cite{AT92}). Indeed, they can be used to prove that even cycles are $2$-list-colorable, while the DP-chromatic number of any cycle is $3$, regardless of its length (see Example~\ref{exmp:cycles}). In this note we demonstrate the failure in the context 
	of DP-coloring of two other list-coloring results whose proofs rely on either the Alon--Tarsi method or the Bondy--Boppana--Siegel lemma.
	
	A well-known application of the orientation method is the following result:
	
	\begin{theo}[{Alon--Tarsi~\cite[Corollary~3.4]{AT92}}]\label{theo:AT}
		Every planar bipartite graph is $3$-list-colorable.
	\end{theo}
	
	We show that Theorem~\ref{theo:AT} does not hold for DP-colorings (note that every planar triangle-free graph is $3$-degenerate, hence $4$-DP-colorable):
	
	\begin{theo}\label{theo:plan_bip}
		There exists a planar bipartite graph $G$ with $\chi_{DP}(G) = 4$.
	\end{theo}
	
	This answers a question of Grytczuk (personal communication, 2016). We prove Theorem~\ref{theo:plan_bip} in Section~\ref{sec:plan_bip}.
	
	Our second result concerns edge colorings. Recall that the \emph{line graph} $\mathsf{Line}(G)$ of a graph $G$ is the graph with vertex set $E(G)$ such that two vertices of $\mathsf{Line}(G)$ are adjacent if and only if the corresponding edges of $G$ share an endpoint. The chromatic number, the list-chromatic number, and the DP-chromatic number of $\mathsf{Line}(G)$ are called the \emph{chromatic index}, the \emph{list-chromatic index}, and the \emph{DP-chromatic index} of $G$ and are denoted by $\chi'(G)$, $\chi'_\ell(G)$, and $\chi'_{DP}(G)$ respectively. The following hypothesis is known as the \emph{Edge List Coloring Conjecture} and is a major open problem in graph theory:
	
	\begin{conj}[{Edge List Coloring Conjecture, see~\cite{JT95}}]\label{conj:ELCC}
		For every graph $G$, $\chi'_\ell(G) = \chi'(G)$.
	\end{conj}
	
	In an elegant application of the orientation method, Galvin~\cite{Gal95}  
	verified the Edge List Coloring Conjecture for bipartite graphs:
	
	\begin{theo}[{Galvin~\cite{Gal95}}]
		For every bipartite graph $G$, $\chi'_\ell(G) = \chi'(G) = \Delta(G)$.
	\end{theo}
	
	We show that this famous result fails for DP-coloring; in fact, it is impossible for a $d$-regular graph $G$ with $d \geq 2$ to have DP-chromatic index $d$:
	
	\begin{theo}\label{theo:reg_ind}
		If $d \geq 2$, then every $d$-regular graph $G$ satisfies $\chi'_{DP}(G) \geq d+1$.
	\end{theo}
	
	We prove Theorem~\ref{theo:reg_ind} in Section~\ref{sec:reg_ind}.
	
	Vizing~\cite{Viz64} proved that the inequality $\chi'(G) \leq \Delta(G) + 1$ holds for all graphs $G$. He also conjectured the following weakening of the Edge List Coloring Conjecture:
	
	\begin{conj}[{Vizing}]\label{conj:viz}
		For every graph $G$, $\chi'_\ell(G) \leq \Delta(G)+1$.
	\end{conj}
	
	We do not know if Conjecture~\ref{conj:viz} can be extended to DP-colorings:
	
	\begin{problem}
		Do there exist graphs $G$ with $\chi'_{DP}(G) \geq \Delta(G) + 2$?
	\end{problem}

	In Section~\ref{sec:multi} we discuss two natural ways to define edge-DP-colorings for multigraphs. According to one of them, the DP-chromatic index of the multigraph $K^d_2$ with two vertices joined by $d$ parallel edges is $2d$.

	\section{Proof of Theorem~\ref{theo:plan_bip}}\label{sec:plan_bip}
	
	In this section we construct a planar bipartite graph $G$ with DP-chromatic number $4$. The main building block of our construction is the graph $Q$ shown in Figure~\ref{fig:cube} on the left, i.e., the skeleton of the $3$-dimensional cube. Let $\Cov{F} = (L, F)$ denote the cover of $Q$ shown in Figure~\ref{fig:cube} on the right.
	
	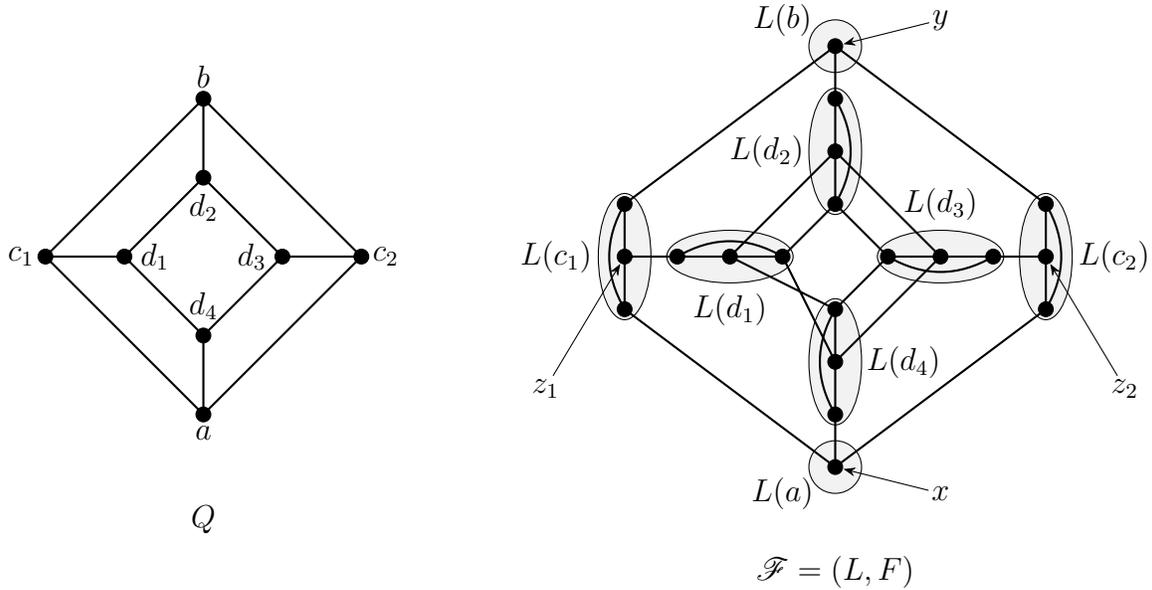
\begin{figure}[h]
		\centering	
		\begin{tikzpicture}[scale=0.7]
		\definecolor{light-gray}{gray}{0.95}

		\foreach \x in {(0,0), (1.5,0), (3,1.5), (3,3), (3,-1.5), (3,-3), (4.5,0), (6,0)}
		\filldraw \x circle (4pt);
		
		\draw[thick] (0,0) node[anchor=east] {$c_1$} -- (1.5,0) -- (3,1.5) -- (3,3) node[anchor = south]{$b$} -- (6,0) node[anchor=west] {$c_2$} -- (4.5,0) -- (3, -1.5) -- (3,-3) node[anchor=north] {$a$} -- cycle -- (3,3)  (3,1.5) -- (4.5,0) (1.5,0) -- (3, -1.5) (3, -3) -- (6,0);
		
		\node at (3, -1.4) [anchor=south] {$d_4$} ;
		\node at (1.6, 0) [anchor=west]{$d_1$};
		\node at (3, 1.4) [anchor=north] {$d_2$};
		\node at (4.4, 0) [anchor=east] {$d_3$};

		\node at (3, -5) {$Q$};
		
		\begin{scope}[xshift=15cm]
		
		\filldraw[fill=light-gray] (-4,0) circle [x radius=5mm, y radius=1.2cm] (4,0) circle [x radius=5mm, y radius=1.2cm] (0,2) circle [x radius=5mm, y radius=1.2cm] (0,-2) circle [x radius=5mm, y radius=1.2cm] (0,4) circle [radius=5mm] (0,-4) circle [radius=5mm] (-2,0) circle [x radius=1.2cm, y radius=5mm] (2,0) circle [x radius=1.2cm, y radius=5mm];
		
		\foreach \x in {1,...,3} 
		\filldraw (\x, 0) circle (4pt) (0, \x) circle (4pt) (-\x, 0) circle (4pt) (0, -\x) circle (4pt);
		\filldraw (-4, -1) circle (4pt) (-4, 1) circle (4pt) (4, -1) circle (4pt) (4, 1) circle (4pt) (-4, 0) circle (4pt) (4, 0) circle (4pt) (0, -4) circle (4pt) (0, 4) circle (4pt);
		
		\draw[thick] (0, -4) -- (-4, -1) (0, -4) -- (4, -1) (0, -4) -- (0, -3);
		\draw[thick] (0, 4) -- (-4, 1) (0, 4) -- (4, 1) (0, 4) -- (0, 3);
		\draw[thick] (-4, 0) -- (-3, 0) (4, 0) -- (3, 0);
		\draw[thick] (-2, 0) -- (0, 2) -- (2, 0) -- (0, -2) -- (-1, 0) -- (0, 1) -- (1, 0) -- (0, -1) -- cycle;
		
		\draw[thick] (-4, -1) -- (-4, 0) -- (-4, 1);
		\draw[thick, bend left] (-4, -1) to (-4,1);
		\draw[thick] (4, -1) -- (4, 0) -- (4, 1);
		\draw[thick, bend right] (4, -1) to (4,1);
		\draw[thick] (-3, 0) -- (-2, 0) -- (-1, 0);
		\draw[thick, bend left] (-3, 0) to (-1,0);
		\draw[thick] (3, 0) -- (2, 0) -- (1, 0);
		\draw[thick, bend left] (3, 0) to (1,0);
		\draw[thick] (0, 1) -- (0, 2) -- (0, 3);
		\draw[thick, bend right] (0, 1) to (0,3);
		\draw[thick] (0, -1) -- (0, -2) -- (0, -3);
		\draw[thick, bend right] (0, -1) to (0,-3);
		
		\node at (0, -6) {$\Cov{F} = (L, F)$};
		
		\node at (-1, -4.5) {$L(a)$};
		\node at (-1, 4.5) {$L(b)$};
		\node at (-5.3, 0) {$L(c_1)$};
		\node at (5.3,0) {$L(c_2)$};
		\node at (-2, -1) {$L(d_1)$};
		\node at (-1.3, 2) {$L(d_2)$};
		\node at (2, 1) {$L(d_3)$};
		\node at (1.3, -2) {$L(d_4)$};
		
		
		\path (2,-4.5) node[inner sep=1pt] (x) {$x$} (0, -4) node[circle, inner sep=0pt, minimum size=6pt] (x1) {};
		\draw[-{Stealth[length=1.6mm]}] (x) to (x1);
		
		\path (2,4.5) node[inner sep=1pt] (y) {$y$} (0, 4) node[circle, inner sep=0pt, minimum size=6pt] (y1) {};
		\draw[-{Stealth[length=1.6mm]}] (y) to (y1);
		
		\path (-5.5,-2.5) node[inner sep=1pt] (z1) {$z_1$} (-4, 0) node[circle, inner sep=0pt, minimum size=6pt] (z11) {};
		\draw[-{Stealth[length=1.6mm]}] (z1) to (z11);
		
		\path (5.5,-2.5) node[inner sep=1pt] (z2) {$z_2$} (4, 0) node[circle, inner sep=0pt, minimum size=6pt] (z21) {};
		\draw[-{Stealth[length=1.6mm]}] (z2) to (z21);
		
		\end{scope}
		\end{tikzpicture}
		\caption{The graph $Q$ (left) and its cover $\Cov{F}$ (right).}
		\label{fig:cube}
	\end{figure}
	
	\begin{lemma}\label{lemma:K}
		The graph $Q$ is not $\Cov{F}$-colorable.
	\end{lemma}
	\begin{proof}
		Suppose, towards a contradiction, that $I$ is an $\Cov{F}$-coloring of $Q$. Since $L(a) = \set{x}$, we have $x \in I$, and, similarly, $y \in I$. Since $z_1$ is the only vertex in $L(c_1)$ that is not adjacent to $x$ or $y$, we also have $z_1 \in I$, and, similarly, $z_2 \in I$. This leaves only $2$ vertices available in each of $L(d_1)$, $L(d_2)$, $L(d_3)$, and $L(d_4)$, and it is easy to see that these $8$ vertices do not contain an independent set of size $4$ (cf.~the cover $\Cov{H}_2$ of the $4$-cycle shown in Figure~\ref{fig:cycle} on the right).
	\end{proof}
	
	Consider $9$ pairwise disjoint copies of $Q$, labeled $Q_{ij}$ for $1 \leq i$, $j \leq 3$. For each vertex $u \in V(Q)$, its copy in $Q_{ij}$ is denoted by $u_{ij}$. Let $\Cov{F}_{ij} = (L_{ij}, F_{ij})$ be a cover of $Q_{ij}$ isomorphic to $\Cov{F}$. Again, we assume that the graphs $F_{ij}$ are pairwise disjoint and use $u_{ij}$ to denote the copy of a vertex $u \in V(F)$ in $F_{ij}$. Let $G$ be the graph obtained from the (disjoint) union of the graphs $Q_{ij}$ by identifying the vertices $a_{11}$, \ldots, $a_{33}$ to a new vertex $a^\ast$ and the vertices $b_{11}$, \ldots, $b_{33}$ to a new vertex~$b^\ast$. Let $H$ be the graph obtained from the union of the graphs $F_{ij}$ by identifying, for each $1 \leq i$, $j \leq 3$, the vertices $x_{i1}$, $x_{i2}$, $x_{i3}$ to a new vertex $x_i$ and the vertices $y_{1j}$, $y_{2j}$, $y_{3j}$ to a new vertex $y_j$. Define the map $L^\ast \colon V(G) \to \powerset{V(H)}$ as follows:
	$$
		L^\ast(u) \defeq \begin{cases}
			L_{ij}(u) &\text{if } u \in V(Q_{ij});\\
			\set{x_1, x_2, x_3} &\text{if } u = a^\ast;\\
			\set{y_1, y_2, y_3} &\text{if } u = b^\ast.
		\end{cases}
	$$
	Then $\Cov{H} \defeq (L^\ast, H)$ is a $3$-fold cover of $G$. We claim that $G$ is not $\Cov{H}$-colorable. Indeed, suppose that $I$ is an $\Cov{H}$-coloring of $G$ and let $i$ and $j$ be the indices such that $\set{x_i, y_j} \subset I$. Then $I$ induces an $\Cov{F}_{ij}$-coloring of $Q_{ij}$, which cannot exist by Lemma~\ref{lemma:K}. Since $G$ is evidently planar and bipartite, the proof of Theorem~\ref{theo:plan_bip} is complete.
	
	\section{Proof of Theorem~\ref{theo:reg_ind}}\label{sec:reg_ind}
	
	Let $d \geq 2$ and let $G$ be an $n$-vertex $d$-regular graph. If $\chi'(G) = d+1$, then $\chi'_{DP}(G) \geq d+1$ as well, so from now on we will assume that $\chi'(G) = d$. In particular, $n$ is even. Indeed, a proper coloring of $\mathsf{Line}(G)$ is the same as a partition of $E(G)$ into matchings, and if $n$ is odd, then $d$ matchings can cover at most $d(n-1)/2 < dn/2 = |E(G)|$ edges of $G$.
	
	Let $uv \in E(G)$  and let $G'\defeq G-uv$. 
 Our argument hinges on the following simple observation:
	
	\begin{lemma}\label{lemma:same}
		Let $C$ be a set of size $d$ and let $f \colon E(G') \to C$ be a proper coloring of $\mathsf{Line}(G')$. For each $w \in \{u,v\}$, let $f_w$ denote the unique color in $C$ not used in coloring the edges incident to $w$.
 Then $f_u =f_v$.
	\end{lemma}
	\begin{proof}
		For each $c \in C$, let $M_c \subseteq E(G')$ denote the matching formed by the edges $e$ with $f(e) = c$. Then $|M_c| \leq n/2$ for all $c \in C$. Moreover, by definition, $\max \set{|M_{f_u}|, |M_{f_v}|} \leq n/2 - 1$. Thus, if $f_u \neq f_v$, then
		$$
			\frac{dn}{2} - 1 = |E(G')| = \sum_{c \in C} |M_c| \leq \frac{dn}{2} - 2;
		$$
		a contradiction.
	\end{proof}
	
	Let $\Z_d$ denote the additive group of integers modulo $d$ and let $H$ be the graph with vertex set $$V(H) \defeq E(G) \times \Z_d,$$ in which the following pairs of vertices are adjacent:
	\begin{itemize}
		\item[--] $(e, i)$ and $(e, j)$ for $e \in E(G)$ and $i$, $j \in \Z_d$ with $i \neq j$,
		\item[--] $(e, i)$ and $(h, i)$ for $eh \in E(\mathsf{Line(G')})$ and $i \in \Z_d$,
		\item[--] $(uv, i)$ and $(uv', i)$ for $uv' \in E(G')$ and $i \in \Z_d$;
		\item[--] $(uv, i)$ and $(u'v, i + 1)$ for $u'v \in E(G')$ and $i \in \Z_d$.
	\end{itemize}
	For each $e \in E(G)$, let $L(e) \defeq \set{e} \times \Z_d$. Then $\Cov{H} \defeq (L, H)$ is a $d$-fold cover of $\mathsf{Line}(G)$. We claim that $\mathsf{Line}(G)$ is not $\Cov{H}$-colorable (which proves Theorem~\ref{theo:reg_ind}). Indeed, suppose that $I$ is an $\Cov{H}$-coloring of $\mathsf{Line}(G)$. For each $e \in E(G')$, let $f(e)$ denote the unique element of $\Z_d$ such that $(e, f(e)) \in I$. Then $f$ is a proper coloring of $\mathsf{Line}(G')$ with $\Z_d$ as its set of colors. Let $f_u$ be the unique element of $\Z_d$ that is not used in coloring the edges incident to $u$. Then the only element of $L(uv)$ that can, and therefore must, belong to $I$ is $(uv, i)$. On the other hand, Lemma~\ref{lemma:same} implies that $i$ is also the unique element of $\Z_d$ that is not used in coloring the edges incident to $v$, and, in particular, for some $u'v \in E(G')$, $f(u'v) = i+1$. Since $(uv, i)$ and $(u'v, i+1)$ are adjacent vertices of $H$, $I$ is not an independent set, which is a contradiction.

	\section{Edge-DP-colorings of multigraphs}\label{sec:multi}

	One can extend the notion of DP-coloring to loopless multigraphs, see~\cite{BKP16}. The definitions are almost identical; the only difference is that in Definition~\ref{defn:cover}, \ref{item:matching} is replaced by the following:

	\begin{enumerate}[labelindent=\parindent,leftmargin=*,label=(C\arabic*)]
		\item[(C4$'$)] If $u$ and $v$ are connected by $t \geq 1$ edges in $G$, then $E_H(L(u), L(v))$ is a union of $t$ matchings.
	\end{enumerate}
	
	An interesting property of DP-coloring of multigraphs is that the DP-chromatic number of a multigraph may be larger than its number of vertices. For example, the multigraph $K^t_k$ obtained from the complete graph $K_k$ by replacing each edge with $t$ parallel edges satisfies
	$$
		\chi_{DP}(K^t_k) = \Delta(K^t_k) + 1 = tk - t + 1.
	$$
	(See~\cite[Lemma~7]{BKP16}.)

	Similarly to the case of simple graphs, the \emph{line graph} $\mathsf{Line}(G)$ of a multigraph $G$ is the graph with vertex set $E(G)$ such that two vertices of $\mathsf{Line}(G)$ are adjacent if and only if the corresponding edges of $G$ share at least one endpoint. Notice that, in particular, $\mathsf{Line}(G)$ is always a simple graph. Sometimes, instead of $\mathsf{Line}(G)$, it is more natural to consider the \emph{line multigraph} $\mathsf{MLine}(G)$, where if two edges of $G$ share both endpoints, then the corresponding vertices of $\mathsf{MLine}(G)$ are joined by a pair edges. Line multigraphs were used, e.g., in the seminal paper by Galvin~\cite{Gal95} and also in~\cite{BKW97, BKW98}.

	Somewhat surprisingly, Shannon's bound $\chi'(G) \leq 3\Delta(G)/2$~\cite{Sha49} on the chromatic index of a multigraph $G$ does not extend to $\chi_{DP}(\mathsf{MLine}(G))$. Indeed, if $G \cong K^d_2$, i.e., if $G$ is the $2$-vertex multigraph with $d$ parallel edges, then $\mathsf{MLine}(G) \cong K^2_d$, so $$\chi_{DP}(\mathsf{MLine}(G)) = \chi_{DP}(K^2_d)= 2d-1 = 2 \Delta(G) -1.$$ This is in contrast with the result in~\cite{BKW97} that $\chi'_\ell(G)\leq 3\Delta(G)/2$ for every multigraph $G$. However, we conjecture that the analog of Shannon's theorem holds for line \emph{graphs}:

	\begin{conj}
		For every multigraph $G$, $\chi_{DP}(\mathsf{Line}(G)) \leq 3\Delta(G)/2$.
	\end{conj}
	
	
\end{document}